\documentclass[a4paper,12pt,twoside]{amsart}
\usepackage[totalwidth=15.75cm,totalheight=22.275cm]{geometry}

\usepackage{LRGmathmacros}

\usepackage{amsmath}
\usepackage{amssymb}
\usepackage{graphicx}

\usepackage[shortlabels]{enumitem}

\usepackage[labelformat=simple]{subcaption}
\usepackage[english]{babel}
\usepackage[T1]{fontenc}
\usepackage[latin1]{inputenc}
\usepackage{hyperref}

\usepackage{color}

\numberwithin{equation}{section}

\theoremstyle{plain}
\newtheorem{theorem}{Theorem}[section]
\newtheorem{proposition}[theorem]{Proposition}
\newtheorem{cor}[theorem]{Corollary}

\newtheorem{lemma}[theorem]{Lemma}

\theoremstyle{definition}
\newtheorem{dfn}[theorem]{Definition}

\theoremstyle{remark} 

\newtheorem{rmk}[theorem]{Remark}

\newcommand{\M}{\mathcal{M}}

\newcommand{\meandering}{J_{\operatorname{m}}}
\newcommand{\radial}{J_{\operatorname{r}}}

\author{Vasiliki Evdoridou}
\author{Lasse Rempe-Gillen}
\address{School of Mathematics \& Statistics \\ The Open University \\
Walton  Hall \\ Milton Keynes MK7 6AA \\ UK}
\email{vasiliki.evdoridou@open.ac.uk}
\address{Dept. of Mathematical Sciences \\
	 University of Liverpool \\
   Liverpool L69 7ZL\\
   UK\\ ORCiD: 0000-0001-8032-8580}
\email{l.rempe@liverpool.ac.uk}
\thanks{The second author was supported by a Philip Leverhulme Prize}
\subjclass[2010]{37F10 (primary), 30D05, 54G15 (secondary)}

\begin{document}
\title{Non-escaping endpoints do not explode}

\begin{abstract}
  The family of exponential maps $f_a(z)= e^z+a$ is of fundamental importance in the study of transcendental dynamics. 
   Here we consider the topological structure of certain subsets of the Julia set $J(f_a)$. 
   When $a\in (-\infty,-1)$, and more generally when $a$ belongs to the Fatou set $F(f_a)$,
  it is known that $J(f_a)$ can be written as a union of \emph{hairs} and \emph{endpoints} of these hairs. 
   In 1990, Mayer proved for  $a\in (-\infty,-1)$ that, while 
    the set of 
    endpoints is totally separated,  
    its union with infinity is a connected set. Recently, 
    Alhabib and the second author extended this result
    to the case where $a \in F(f_a)$, and 
    showed that it holds even
    for the smaller set of all 
    \emph{escaping}
    endpoints.

    We show that, in contrast, the set of \emph{non-escaping} endpoints together with infinity is totally separated. 
    It turns out that this property is closely related to a topological structure known as a `spider's web'; 
    in particular we give a 
    new topological characterisation of spiders' webs that may be of independent interest. We also show how our results
    can be applied to \emph{Fatou's function}, $z\mapsto z + 1 + e^{-z}$. 
\end{abstract}
\maketitle

\section{Introduction}
  The iteration of transcendental entire functions was initiated by Fatou in 1926~\cite{fatou}, and has received
    considerable attention in recent years. The best-studied examples are provided
     by the functions 
            \begin{equation}\label{eqn:exponentialmaps} f_a \colon \C\to \C; \quad z\mapsto e^z+a \end{equation}
    for $a\in (-\infty,-1)$. (See \cite{devaney-84,devaney-krych84,aartsoversteegen,karpinskaparadox}.) In this case, there is a unique attracting fixed point $\zeta$ on 
    the negative real axis. The (open) set of starting values whose orbits under $f_a$ converge to  $\zeta$ under
    iteration is connected and dense in the plane. (See Figure~\ref{subfig:disjointtype}.) 

   The complement of this basin of attraction, the \emph{Julia set} $J(f_a)$, is known 
   to be an uncountable union of pairwise disjoint arcs, known as ``hairs'',
   each of which joins a finite \emph{endpoint} to infinity.
    More precisely, $J(f_a)$ is what is known as a ``Cantor bouquet''; see  
    \cite{aartsoversteegen} and \cite{baranski-jarque-rempe12} for further information. The action of $f_a$ on $J(f_a)$ 
    provides the simplest (yet far from simple) transcendental entire dynamical system. 
    Results first established in this context have often led to an increased understanding 
    in far more general settings. 

\begin{figure}
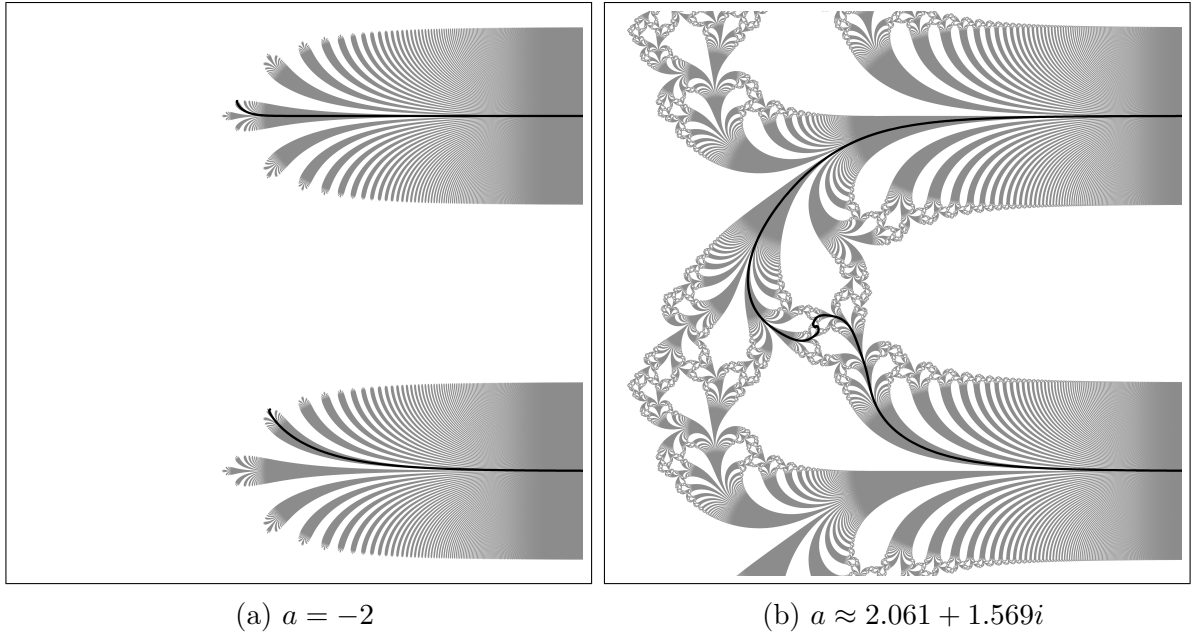

  \begin{subfigure}[b]{.5\linewidth}
    \fbox{%
    \includegraphics[width=.95\textwidth]{Disjoint-Type-Exponential-rays-v2}}
    \caption{$a=-2$}\label{subfig:disjointtype}
  \end{subfigure}
\hfill
  \begin{subfigure}[b]{.5\linewidth}
    \fbox{%
    \includegraphics[width=.95\textwidth]{Expo_Per3_rays-v2}}
    \caption{$a\approx 2.061 + 1.569i$}\label{subfig:per3}
  \end{subfigure}
  \caption{Two maps in the exponential family; the Julia set is shown in grey, while in each picture two individual hairs are shown in black. The white regions correspond to
     points which converge to an attracting periodic cycle under iteration. 
      In~\subref{subfig:disjointtype}, the Julia set is a Cantor bouquet, and different hairs have different endpoints. The map in~\subref{subfig:per3} has an attracting cycle of period $3$; several hairs share
      the same endpoint, as is the case for the two hairs shown.}
\end{figure}

  A topological model of $f_a|_{J(f_a)}$ was given in \cite{aartsoversteegen} in terms of a \emph{straight brush}, where the hairs of $J(f_a)$ are represented by straight horizontal
    rays. This model depended a priori
   on the parameter $a\in (-\infty, -1)$, but a natural version of the construction
   that is independent of $a$ was given in~\cite{Rem06}. 
   From this point of view,
    the topological dynamics in this case is completely understood, but the set $J(f_a)$ nonetheless exhibits a number of subtle and surprising properties (compare \cite{karpinskaparadox} for
    a celebrated and unexpected result concerning Hausdorff dimensions). 

  In particular, Mayer \cite{mayer90} proved in 1990 that the set $E(f_a)$ of endpoints of $f_a$ (see above) has the 
   intriguing property that
    $E(f_a)\cup\{\infty\}$ is connected, while 
   $E(f_a)$ itself is totally separated. 
    Here a totally separated space is defined as follows.
\begin{dfn}[Separation]\label{defn:separation}
 Let $X$ be a topological space. Two points $a, b \in X$ are \emph{separated} (in $X$) if there is an open and closed   
    subset $U \subset X$ with $a \in U$ and $b \notin U.$  If every pair of points in $X$ is separated we say that $X$ is a 
    \emph{totally separated} space.
\end{dfn}
If $X$ is connected but $X \setminus \{x_0\}$, $x_0 \in X$, is totally separated, then we say that $x_0$ is an \emph{explosion point} of $X$. 
   Hence, infinity is an explosion point for $E(f_a) \cup \{\infty\}, a<-1$. Following the terminology used in \cite{alhabib-rempe15}, we will also 
    simply say that infinity is an explosion point for $E(f_a)$.

Alhabib and the second author recently proved \cite[Theorem 1.3]{alhabib-rempe15}, that Mayer's result holds also for the smaller set of 
  \emph{escaping} endpoints of $J(f_a)$; that is, for the set $\tilde{E}(f_a)\defeq E(f_a)\cap I(f_a)$ of endpoints that belong to the escaping set 
\[ I(f_a)= \{z\in\C\colon  f_a^n(z) \to \infty\;\text{as}\;n \to \infty\}. \]
   Still assuming that $a\in (-\infty,-1)$, the complementary set of \emph{non-escaping endpoints} of $f_a$ satisfies the following
     identities:  
     \[ E(f_a) \setminus \tilde{E}(f_a) = J(f_a) \setminus I(f_a) = \radial(f_a); \]
    see Corollary~\ref{cor:radialJ} and Proposition~\ref{prop:hairs}. 
   Here $\radial(f_a)$ is the \emph{radial Julia set}, a set of particular importance.
   The results of \cite{alhabib-rempe15} naturally suggest the question whether $\infty$ 
    is an explosion point for $\radial(f_a)$ also. It is known that 
    $\radial(f_a)$ has Hausdorff dimension strictly greater than one \cite[Section~2]{urbanski-zdunik03}, which is compatible with this possibility. Nonetheless,
   we prove here that 
    the sets of escaping and non-escaping endpoints are topologically very different from each other.

 \begin{theorem}[Non-escaping endpoints do not explode]\label{thm:maindisjointtype}
    Let $a\in (-\infty,-1)$. Then $\radial(f_a)\cup \{\infty\}$ is totally separated. 
 \end{theorem}

 The maps $f_a$ as in~\eqref{eqn:exponentialmaps}, for general  $a\in\C$, have also been investigated
    in considerable detail; compare e.g.\
    \cite{baker-rippon,devaney-krych84,schleicher-zimmer03,Rem07,rempe-schleicher09}. This family can be
    considered as an analogue of the family of quadratic polynomials, which gives rise to the famous Mandelbrot set
    (see \cite{BDHRGH} and \cite{rempe-schleicher08}). 

  The dynamical structure of $f_a$ for general $a$ is much more complicated than for real $a<-1$, but
    nonetheless Theorem~\ref{thm:maindisjointtype} can be extended to a large and conjecturally dense set of parameters. 
     More precisely, it is known \cite{Rem07} that the \emph{Julia set} $J(f_a)$ (informally, the locus of
    chaotic dynamics; see Section~\ref{sec:preliminaries}) can be written as a union of hairs and endpoints if and only if $a\notin J(f_a)$.
    In this case, $a$ belongs to the \emph{Fatou set} 
     $F(f_a)=\C\setminus J(f_a)$, and $F(f_a)$ is precisely the basin of attraction of an attracting or parabolic periodic cycle (see Proposition~\ref{prop:attractingparabolic}). 
    Note that $F(f_a)$  is no longer
     connected in general; compare Figure~\ref{subfig:per3}. In particular, $J(f_a)$ is not a Cantor bouquet in this 
   case, but rather a more complicated structure where different curves share the same endpoint, 
    known as a \emph{pinched Cantor bouquet}. (For further details, see the discussion in  
   Section~\ref{sec:preliminaries}.)

  All of the above  discussion for the case $a<-1$ carries over to the case where $a \in F(f_a)$, with the exception that the radial Julia set is a proper subset of the set of non-escaping points when $F(f_a)$ has a parabolic cycle. By~\cite{alhabib-rempe15}, 
    the set $\tilde{E}(f_a)$ has $\infty$ as an explosion point; in contrast, we again show that $(J(f_a)\setminus I(f_a) ) \cup \{\infty\}$ is totally separated also in this case.

   A second strengthening of Theorem~\ref{thm:maindisjointtype} 
    arises by considering not only the division of the set $E(f_a)$  into escaping and non-escaping endpoints, but also by distinguishing escaping points
    by the \emph{speed} of their escape.  The \emph{fast} escaping set $A(f)$ of a transcendental entire function $f$, introduced by 
     Bergweiler and Hinkkanen \cite{bergweiler-hinkkanen99} and 
    investigated closely by Rippon and Stallard (see \cite{rippon-stallard-05,rippon-stallard12}), has played an important role in recent progress in transcendental dynamics. 
   Informally, $A(f)$ consists of those points
    of $I(f)$ that tend to infinity at the fastest rate possible; for a formal definition, see~\eqref{eqn:fastescaping} below. Let us say that a point $z\in J(f)$ is \emph{meandering} if
    it does not belong to $A(f)$, and denote the set of meandering points by $\meandering(f)$. It is known that for $f_a$, $a \in \C$, every point on a hair belongs to $A(f_a)$ (see \cite[Lemma~5.1]{schleicher-zimmer03} and compare also 
    \cite{alhabib-rempe15}); that is, when $a\in F(f_a)$ 
    every meandering
    point is an endpoint. 

  By \cite[Remark on p.\ 68]{alhabib-rempe15}, for $a$ as above, infinity is an explosion point even for 
    the set $E(f_a)\cap A(f_a)$ of fast escaping endpoints. In contrast, we show the following.

 \begin{theorem}[Meandering endpoints do not explode] \label{thm:main}
    Suppose that $a\in\C$ is such that $a\in F(f_a)$. Then the set $\meandering(f_a)\cup\{\infty\}$ is totally separated. 
 \end{theorem}   

\subsection*{Spiders' webs}

 Our results have a connection with a topological structure introduced by Rippon and Stallard in \cite{rippon-stallard12},
   known as a `spider's web'.

\begin{dfn}[Spider's web]\label{defn:spidersweb}
A set $E \subset \mathbb{C}$ is an \emph{(infinite) spider's web} if $E$ is connected and there exists a sequence $(G_n)_{n=0}^{\infty}$ of bounded simply connected domains, with $G_n \subset~ G_{n+1}$ 
    and  $\partial G_n \subset E$ for $n\geq 0$, such that $\bigcup_{n=0}^{\infty} G_n= \mathbb{C}.$ 
\end{dfn}

We shall prove that the complement $A(f_a) \cup F(f_a)$ of $\meandering(f_a)$ is a spider's web when $a \in F(f_a)$.
  Together with the fact that  $\meandering(f_a)\subset E(f_a)$ itself is totally separated (in $\C$), this easily
  implies Theorem~\ref{thm:main}. Our proof uses a new topological characterisation of spiders' webs, which
  may have independent interest.

\begin{theorem}[Characterisation of spiders' webs]\label{thm:spiderswebs}
  Let $E\subset\C$ be connected. Then  $E$ is a spider's web if and only if  $E$ separates every  finite point
    $z\in\C$ from  $\infty$. 

   (Here $E$ separates $z$ from $\infty$ if the two points are separated in
      $(\C\setminus E)\cup \{z,\infty\}$.) 
\end{theorem}

\subsection*{Singular values in the Julia set}

  As mentioned above, when $a\in J(f_a)$, it is no longer possible to write $J(f_a)$ as a union of hairs and 
   endpoints, so questions concerning the structure of the set of endpoints are less natural in this setting.
   (The set of \emph{escaping endpoints}, on the other hand, does remain a natural object; compare the discussion 
    in~\cite{alhabib-rempe15}.) 

   However, the radial Julia set, the set of non-escaping points and the set of meandering points
   remain of interest, and it turns out that their structure usually differs dramatically from the case where $a\in F(f_a)$. 
   Indeed, consider the \emph{postsingularly finite} case, where the omitted value $a$ eventually
   maps onto a repelling periodic cycle. Then $J(f_a)=~\C$, and $\radial(f_a) = \C\setminus I(f_a)$ (see Corollary \ref{cor:radialJ}(c)). It follows from results of~\cite{Rem07} that
   $\radial(f_a)$ contains
   a dense collection of unbounded connected sets.  In particular,  $\radial(f_a)\cup\{\infty\}$ is connected, and 
   $I(f_a)$ is not a spider's web. (See Proposition~\ref{prop:hairs} and Corollary~\ref{cor:largemeandering}.) It is plausible that $\C\setminus~I(f_a)~\cup \{\infty\}$ is connected for all $a\in\C$, 
   and that, in particular,  $I(f_a)$ and $A(f_a)$ are never spiders' webs. 

  The techniques from our proof of Theorem~\ref{thm:main} can nonetheless be adapted to 
    yield a slightly technical result about the set of meandering points whose orbits stay away from the
    singular value (Theorem~\ref{thm:avoidsingularvalue}). In particular, we recover a result from~\cite{rempe-04}, concerning
    a question of Herman, Baker and Rippon about 
    the boundedness of certain Siegel discs in the exponential family (Theorem~\ref{thm:siegel}). 

  We shall leave open the question whether, when $a\in J(f_a)$, infinity can be an explosion point for the set 
     $\meandering(f_a) \cap I(f_a)$ of points that are meandering and escaping. 

\subsection*{Fatou's function}

Our proofs build on an idea from recent work of the first author \cite{Evdoridou16} concerning Fatou's function 
   \[ f\colon \C\to \C; \qquad z\mapsto z+1+e^{-z}. \]
  For this function, the Julia set is once again a Cantor bouquet, but in contrast to the exponential family
   the Fatou set here is contained in the escaping set  $I(f)$. It is shown in \cite{Evdoridou16} that $I(f)$ is
   a spider's web, and that the set of non-escaping endpoints of $f$ together with infinity is 
    \emph{totally disconnected}. We shall show that Theorem~\ref{thm:maindisjointtype}
    implies the stronger result that $A(f)\cup F(f)$ is a spider's web for Fatou's function, and that
    the set of meandering endpoints $\meandering(f)\subset J(f)$ together with infinity is a totally separated set. 
    In particular, this illustrates that our
    results have consequences beyond the exponential family itself.

\subsection*{Structure of the article} In Section~\ref{sec:preliminaries}, we review key definitions and facts from exponential dynamics. We also
    establish Theorem~\ref{thm:spiderswebs}, concerning spiders' webs. Our results about exponential maps $f_a$ with $a\in F(f_a)$ are proved in
    Section~\ref{section_exp}, while the case where $a\in J(f_a)$ is discussed in Section~\ref{sec:julia}. Finally, we discuss Fatou's function in Section~\ref{section_fatou}. 

\subsection*{Acknowledgements} 
 We thank Phil Rippon, Dave Sixsmith and Gwyneth Stallard for interesting discussions and comments on this work. We are also grateful to Dave Sixsmith for detailed
    comments and feedback that helped to improve the presentation of the paper.  

\section{Preliminaries} \label{sec:preliminaries}

 \subsection*{Notation and background}
   We denote the complex plane by $\mathbb{C}$ and the Riemann sphere by $\Ch= \mathbb{C} \cup \{\infty\}$. The round disc of radius $r$ around a point $z_0$ is denoted by
     $D(z_0,r)$. 

 Let $f$ be a transcendental entire function. As noted in the introduction, the \emph{Julia set} and \emph{Fatou set}
    of $f$ are denoted $J(f)$ and  $F(f)$, respectively. Here $F(f)$ is the set of normality of the family of 
    iterates $(f^n)_{n=1}^{\infty}$ of $f$, and  $J(f)=\C\setminus  F(f)$ is its complement. For further background
    on transcendental dynamics, we refer to \cite{Berg}. 

The \emph{radial Julia set} $\radial(f)\subset J(f)$ can be described as the locus of non-uniform hyperbolicity: at a point $z\in \radial(f)$, it is possible to pass from arbitrarily small scales to a definite scale using univalent iterates. 
   More formally, there is a number $\delta>0$ and infinitely many $n$ such that the spherical disc
     of radius $\delta$ around $f^n(z)$ can be pulled back univalently along the orbit. The radial Julia set
     was introduced by Urba{\'n}ski \cite{urbanskiconical} and McMullen~\cite{mcmullenradial} 
     for rational functions. As far as we are aware, its first appearance in the entire setting, in the special case of 
     hyperbolic exponential maps, is due to Urba\'nski and Zdunik \cite{urbanski-zdunik03}. We
   refer to \cite{Rem09} for a general discussion. In the cases of interest to us, the following properties are sufficient to
   characterise the radial Julia set.
   \begin{enumerate}[(1)]
     \item $\radial(f)\subset J(f)\setminus I(f)$.\label{item:radialescaping}
     \item $\radial(f)$ is forward-invariant: $f(\radial(f))\subset \radial(f)$. Furthermore, $\radial(f)$ is almost
         backwards-invariant except at critical values: 
         if $z\in J(f)$ and $f(z)\in \radial(f)$, then either $z$ is a critical point  of $f$ or
         $z\in \radial(f)$.\label{item:radialinvariance}
     \item Every repelling periodic point and no parabolic periodic point belongs to $\radial(f)$.\label{item:radialperiodic}
     \item Suppose that the forward orbit of $z\in J(f)$ 
       has a finite accumulation point that is not in the closure of the union of the forward orbits of
       critical and asymptotic values. Then $z\in \radial(f)$.\label{item:radialpostsingular}
   \end{enumerate}

The \emph{fast escaping set} $A(f)$ plays a key role in our arguments. 
    We shall use the definition given by Rippon and Stallard in \cite{rippon-stallard12},
   which is slightly different, but equivalent, to the original formulation from \cite{bergweiler-hinkkanen99}. 
    Let $f$ be a transcendental entire function.
   For $r>0$, define the maximum modulus 
\[
    M(r,f)= \max_{\lvert z\rvert =r} \lvert f(z)\rvert.
\]
  We also denote the $n$-th iterate of the function $r\mapsto M(r,f)$ by $M^n(\cdot , f)$. 

  Since $f$ is non-linear, we have
     \begin{equation}\label{eqn:R0}
        R(f)\defeq \inf\{R\geq 0\colon M(r,f)>r \text{ for all } r\geq R\} < \infty.
     \end{equation}
     For $R> R(f)$,
   define 
\begin{equation}\label{eqn:fastescaping}
   A(f) \defeq \{z\in\C\colon \;\text{there exists}\;\ell \in \mathbb{N}\;\text{such that}\; \lvert f^{n+\ell}(z) \rvert \geq M^n(R,f),\;\text{for}\;n \in \mathbb{N}\}. \end{equation}
  It can be shown that the definition is independent of $R$. Again following \cite{rippon-stallard12}, we also define
 \[ A_R(f)\defeq \{z\colon\lvert f^n(z) \rvert \geq M^n(R,f),\;\text{for}\;n \in \mathbb{N}\}. \] 

 \subsection*{Attracting and parabolic exponential maps}
  Let us now turn to background results concerning exponential maps $f_a$ as in~\eqref{eqn:exponentialmaps}.
  Recall that our main result concerns the case where $a \in F(f_a)$. 
 
\begin{proposition}[Attracting and parabolic exponential maps] \label{prop:attractingparabolic}
  Let $a\in\C$ and set $f_a(z)\defeq e^z+a$. Then $a\in F(f_a)$ if and only if $f_a$ has an attracting or parabolic
   cycle.

   In  this case, $F(f_a)$ is precisely the basin of attraction of this cycle, 
     and $\infty$ is accessible from the connected component $U_1$ of $F(f_a)$ containing $a$, 
     by an arc along which the real part tends to infinity.
\end{proposition}
\begin{proof}
   This is well-known; see Section~9 of \cite{baker-rippon}, particularly 
     Corollary~2, Theorems~9 and~10, and the two paragraphs following Theorem~10. Accessibility is not mentioned
     explicitly, but follows from the argument for
     unboundedness of Fatou components; compare also \cite[Section~2]{bhattacharjee-devaney-2000} or 
      \cite[Section~4.4]{schleicher-zimmer03-periodic}. In the case where
       $f_a$ has an attracting fixed point, 
       accessibility of infinity (and much more) can be found in the 
       work of Devaney and Goldberg~\cite[Section~4]{DevaneyGoldberg1987}. 
      As the proof of Proposition~\ref{prop:attractingparabolic} will be instructive for our later 
        constructions, and since we are not aware of a reference containing the statement 
        in the precise form that we will use it, we shall provide the details for completeness.

  For all $a\in\C$, the Fatou set 
   $F(f_a)$ does not have wandering components \cite[Theorem~12]{Berg}, nor
    components on which the iterates tend to infinity \cite[Theorem~15]{Berg}.
    Thus the only possible components of $F(f_a)$ are immediate 
    attracting or parabolic basins, Siegel discs, and their iterated preimages \cite[Theorem~6]{Berg}.
   Now suppose that $a\in F(f_a)$, so 
    the forward orbit of $a$ either converges to an attracting or periodic 
    orbit, or is eventually contained in an invariant circle in a Siegel disc. In particular, the set of accumulation points
    of this orbit that belong to $J(f_a)$ is either empty or consists of a single parabolic cycle. 

    By \cite[Theorem~7]{Berg}, 
     the boundary of any Siegel disc is contained in the closure of the orbit of the omitted value $a$,
    and the basin of any attracting or parabolic cycle must contain $a$. The former is impossible by the above,
     so we conclude that $F(f_a)$ consists of the basin of a single attracting or parabolic cycle.

  Let $p$ be the period of the attracting or parabolic cycle and $U_1$ the component of $F(f_a)$ containing $a$. Consider the set 
   $U_0\defeq f_a^{-1}(U_1)$; since $U_1$ contains a neighbourhood of $a$, $U_0$ contains a left half-plane $\mathcal{H}$. 
   Since $f_a\colon U_0 \to U_1 \setminus \{a\}$ is a covering map, any connected component of $U_0$ intersects $\mathcal{H}$, and hence $U_0$ is connected. 
   It follows that $U_0$ lies on the periodic cycle of Fatou components of $U_1$, with $f_a^{p-1}(U_0)=U_1$.   
    Set $n\defeq \max(1,p-1)$, and let $\Gamma$ be a
    connected component of $f_a^{-n}(\Gamma_0)$ contained in $U_1$; then $\Gamma$ satisfies the properties claimed in the proposition. See Figure~\ref{subfig:curves}. 
\end{proof}

\begin{figure}
  \begin{subfigure}[b]{.5\linewidth}
     \fbox{%
    \def\svgwidth{.93\textwidth}
     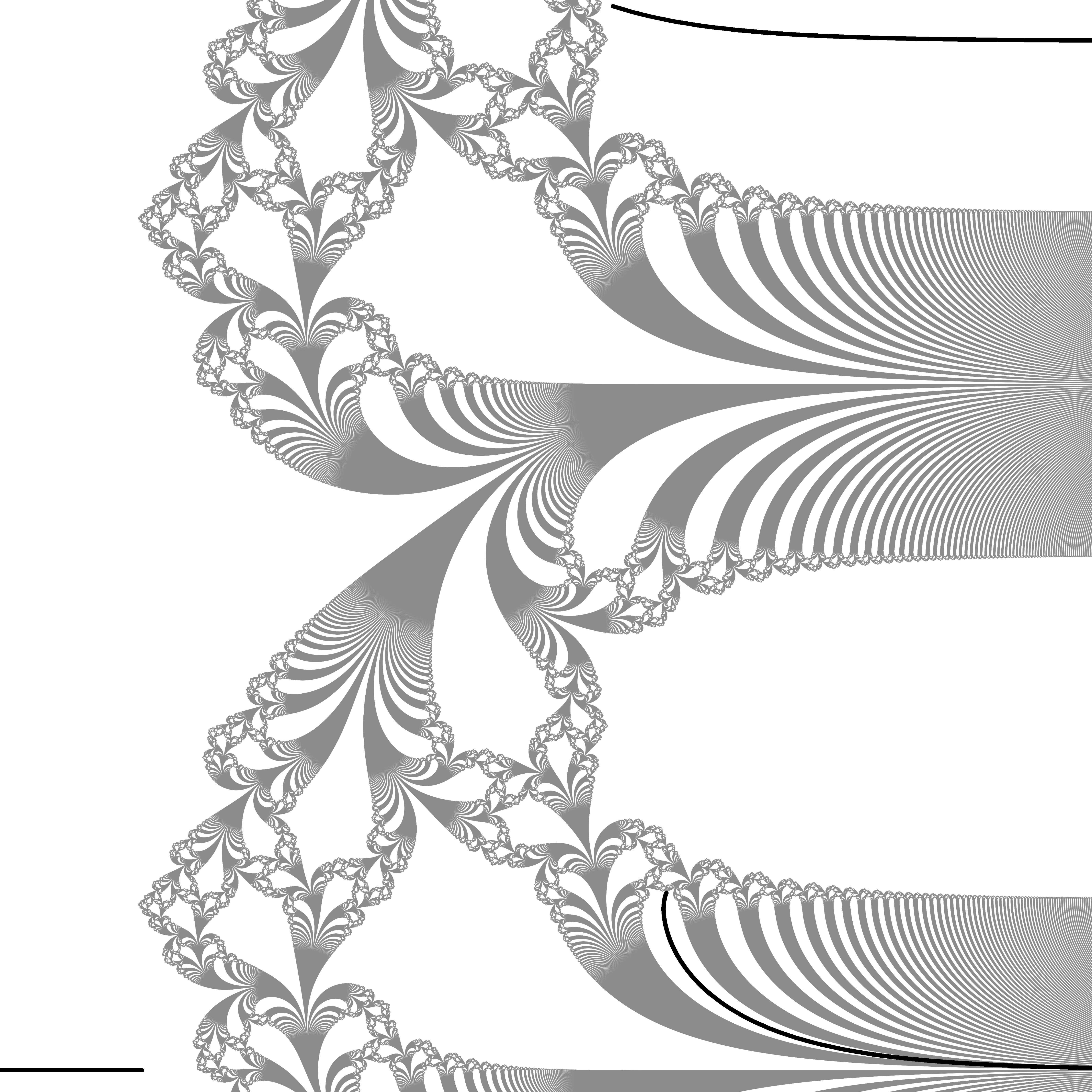}
    \caption{Proposition~\ref{prop:attractingparabolic}}\label{subfig:curves}
  \end{subfigure}
    \hfill
  \begin{subfigure}[b]{.5\linewidth}
     \fbox{%
    \def\svgwidth{.93\textwidth}
     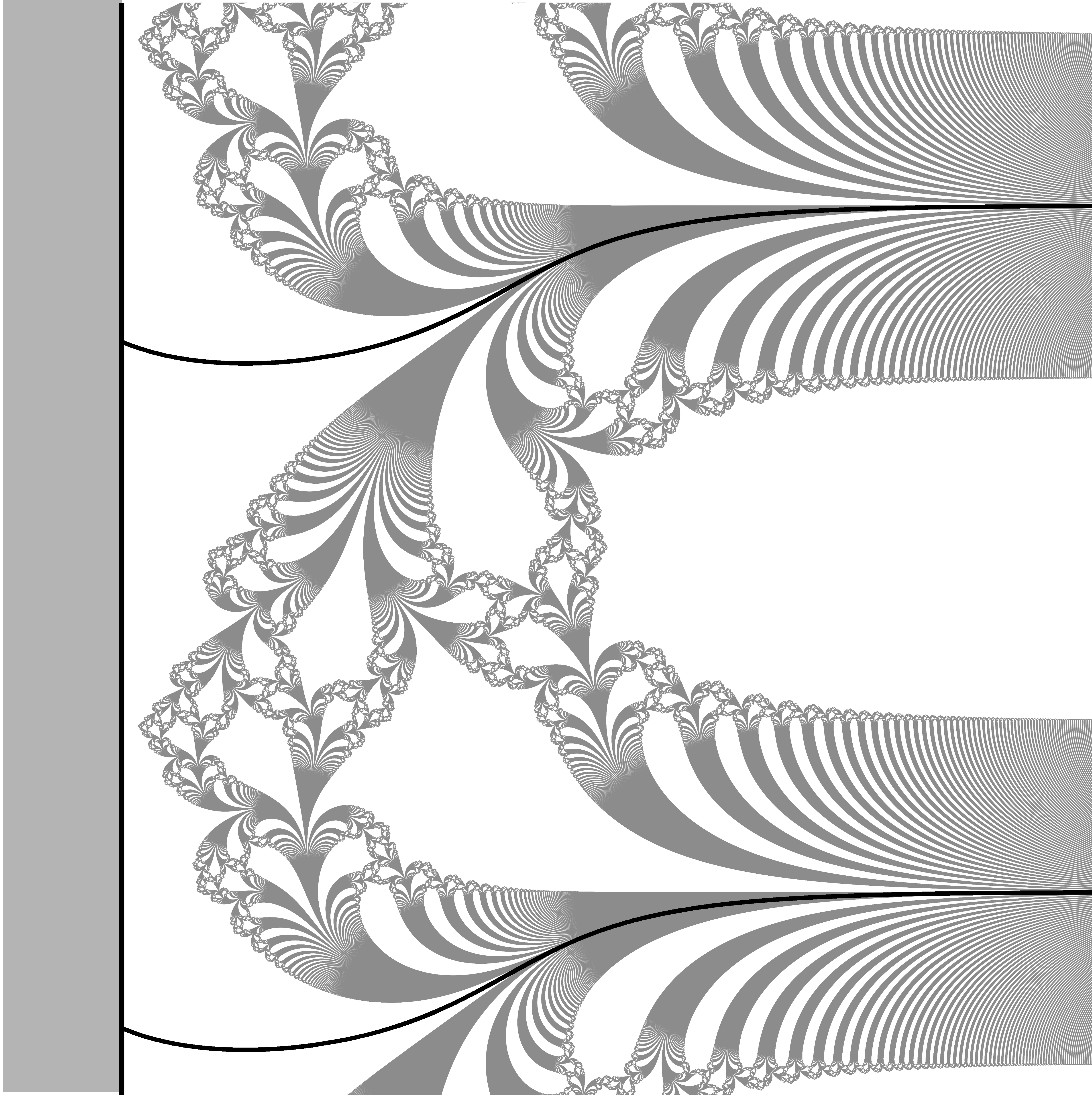}
    \caption{The set $\mathcal{M}$ in the proof of Theorem~\ref{thm:fastseparation}}\label{subfig:M}
  \end{subfigure}
  \caption{Illustration of the proofs of Proposition~\ref{prop:attractingparabolic} and Theorem~\ref{thm:fastseparation},
     using the exponential map $f=f_a$ from  Figure~\ref{subfig:per3}. As shown in~\subref{subfig:curves}, the curve 
     $\Gamma$ is constructed as a pullback of $\Gamma_0$, which is a piece of the negative real axis. 
     The set $\mathcal{M}$ in~\subref{subfig:M} is the union of the left half-plane $\mathcal{H}$ and the set
      $f^{-1}(\sigma)$, which is a union of infinitely many arcs to $\infty$. (Here, $\sigma$ is a subset of the
       curve $\Gamma$ from~\subref{subfig:curves}.)}
\end{figure}

 In particular, we obtain the following description of the radial Julia set in the case where $a\in F(f_a)$ (as well as when
   $f_a$ is postsingularly finite).
 
\begin{cor}[Radial Julia sets]\label{cor:radialJ}
Let $f_a(z) = e^z+a$, $a \in \C$.
\begin{itemize}
\item[\emph{(a)}] If $f_a$ has an attracting periodic point, then $\radial(f_a) = J(f_a)\setminus I(f_a)$;
\item[\emph{(b)}] if $f_a$ has a parabolic periodic point $z_0$, then $\radial(f_a) = J(f_a) \setminus (I(f_a) \cup O^-(z_0))$, where $O^-(z_0)$ denotes the full backward orbit of $z_0$;
\item[\emph{(c)}] if the orbit of $a$ under $f_a$ is finite, then $\radial(f_a) = \C\setminus I(f_a)$.
 \end{itemize}     
 \end{cor}    
\begin{proof}
   In each case, this follows from the properties~\ref{item:radialescaping} to~\ref{item:radialpostsingular} 
   of $\radial(f_a)$ stated at the beginning of the section. Recall
    that always $\radial(f)\subset J(f)\setminus I(f)$ by~\ref{item:radialescaping}. For the
    first two cases, recall the proof of Proposition~\ref{prop:attractingparabolic}. 
    If $f_a$ has an attracting periodic cycle, then the orbit of $a$ is compactly contained in $F(f_a)$, and hence
    every non-escaping point belongs to the radial Julia set by~\ref{item:radialpostsingular}. 
    Now suppose that $f_a$ has a parabolic periodic point $z_0$; then $O^-(z_0)\cap \radial(f_a)=\emptyset$
     by~\ref{item:radialinvariance} and~\ref{item:radialperiodic}. On the other hand, 
    the orbit of $a$ is in the Fatou set and accumulates only on the cycle of $z_0$. Any orbit in 
     $J(f_a)\setminus (I(f_a)\cup O^-(z_0))$ must accumulate at some finite point that is not on this cycle,
     and thus belongs to $\radial(f_a)$ by~\ref{item:radialpostsingular}. 

   Finally, if the orbit of $a$ is finite, then (similarly as in the first part of the proof of
   Proposition~\ref{prop:attractingparabolic}) $J(f_a)=\C$ and $a$ eventually maps to a repelling periodic 
   point $z_0$. As in the parabolic case, we have $J(f_a)\setminus  (I(f_a)\cup O^-(z_0)) \subset \radial(f_a)$,
    but here also $O^-(z_0)\subset \radial(f_a)$ by~\ref{item:radialinvariance} and~\ref{item:radialperiodic}. 
    So  $\radial(f_a) = J(f_a)\setminus I(f_a) = \C\setminus I(f_a)$, as claimed.
\end{proof}

\subsection*{Hairs and endpoints}
  The escaping set of any exponential map $f_a$ decomposes into \emph{hairs} (or \emph{dynamic rays}) and
    (escaping) \emph{endpoints} \cite{schleicher-zimmer03}. These concepts can be defined conveniently as follows~\cite[Definition~1.1]{alhabib-rempe15}.

 \begin{dfn}[Hairs and endpoints]
Let $f_a(z)= e^z+a$, $a \in \C.$ We say that a point $z_0 \in \mathbb{C}$ is on a \emph{hair} if there exists an arc $\gamma\colon [-1,1] \to I(f_a)$ such
that $\gamma(0) = z_0.$ A point $z_0 \in \mathbb{C}$ is an \emph{endpoint} if $z_0$ is not on a hair and there is an arc
$\gamma\colon [0,1]\to \mathbb{C}$ such that $\gamma(0) = z_0$ and $\gamma(t) \in I(f_a)$ for all $t > 0.$
   We denote the set of endpoints by  $E(f_a)$. 
\end{dfn} 

 We shall use the following properties of hairs and endpoints. Note that these are required only for background,
    and to be able to phrase
    our main results in the language of endpoints.

\begin{proposition}[Hairs and endpoints]\label{prop:hairs}
  Let $f_a = e^z+a$, $a\in\mathbb{C}$. 
   \begin{enumerate}[(a)]
     \item Every point in $I(f_a)$ is on a hair or an endpoint.
     \item If $z\in I(f_a)$ is on  a hair, then $z\in A(f_a)$. 
      \item If $a\in F(f_a)$, then every point $z\in  J(f_a)$ is on a hair or an endpoint.%
        \label{item:pinchedcantorbouquet}
     \item If $a\in  J(f_a)$, then there is a point $z\in J(f_a)$ that is neither on a hair nor an 
        endpoint.\label{item:nonlanding1}
     \item If $a$ is on a hair or an endpoint, then $J(f_a)\setminus I(f_a)$ contains a 
        dense collection of unbounded connected sets. 
       \label{item:nonlanding2}
    \end{enumerate}
\end{proposition}
\begin{proof}
  The first claim follows from \cite[Theorem~6.5]{schleicher-zimmer03}. The second
    follows from \cite[Proposition~4.5]{schleicher-zimmer03}, together with well-known
    estimates on exponential growth (see Lemma~\ref{lem:exponentialgrowth}) and the classification
    of path-connected components of $I(f_a)$ \cite[Corollary~4.3]{frs}. Compare also 
    \cite[Section~4]{alhabib-rempe15}. 

 When $f_a$ has an attracting periodic orbit, it was first 
   shown in \cite{bhattacharjee-devaney-2000} that the Julia set of $f_a$ is a ``bouquet''
     of hairs, where different hairs may share the same endpoint. This establishes~\ref{item:pinchedcantorbouquet}
     in this case. In \cite[Corollary~9.3]{Rem06}, the stronger statement is proved that $J(f_a)$ is a 
     \emph{pinched Cantor bouquet}. More precisely, the dynamics of $f_a$ on its Julia set can be 
     described as a quotient of that of $f_{\tilde{a}}$ on $J(f_{\tilde{a}})$, where $\tilde{a}\in (-\infty,-1)$, 
     by an equivalence relation on the endpoints of $J(f_{\tilde{a}})$. 

  As stated in~\cite[Remark on p.\ 1967]{Rem06}, these results also hold for exponential maps having
    a parabolic orbit, establishing~\ref{item:pinchedcantorbouquet}. However, the details of the proof are
    omitted in~\cite{Rem06}; they can be found in forthcoming work of M.\ Alhamd, which treats the general
    case of parabolic entire functions. 

  Finally, let us turn to~\ref{item:nonlanding1} and~\ref{item:nonlanding2}, 
    where $a\in J(f_a)$. If $a$ is neither on a hair nor an endpoint, then
    there is nothing to prove. Otherwise, by 
    \cite{Rem07}, there exists $\gamma\colon (0,\infty)\to I(f_a)$ with $\gamma(t)\to \infty$ as  $t\to\infty$,
    such that the closure $\hat{\gamma}$ of $\gamma$ in $\Ch$ is an indecomposable continuum. Moreover,
    the construction ensures that $\hat{\gamma}$ does not  separate the sphere, that 
    $\hat{\gamma}\cap I(f_a)=\gamma$, and that $a\notin \hat{\gamma}$. 

   As an indecomposable continuum, $\hat{\gamma}$ has uncountably many composants 
     \cite[Theorem~11.15]{nadler}, each of which is
     dense in $\hat{\gamma}$. Exactly one
     of these components contains $\gamma$, and by the above no other composant intersects the 
     escaping set. Thus every other composant is an unbounded connected subset of $J(f_a)\setminus I(f_a)$. 

   By \cite{Mazurkiewicz1929}, the union of composants of $\hat{\gamma}$ that contain a point accessible via 
    a curve in $\Ch\setminus \hat{\gamma}$ is of first category in $\hat{\gamma}$.  In particular, $\hat{\gamma}$ 
    contains a point that is not accessible from $I(f_a)$, proving~\ref{item:nonlanding1}.
   (It also follows directly from the construction in \cite{Rem07} that no point of $\hat{\gamma}$ is 
    an endpoint.)

  Moreover, let $\tilde{\gamma}$ be a connected component of $f_a^{-1}(\gamma)$; then it follows that 
     $\tilde{\gamma}$ has the same properties as $\gamma$. Since iterated preimages of  $\gamma$ are
     dense in  $J(f_a)$,  and unbounded connected sets of non-escaping points are dense in the closure of each
     preimage, the proof is complete.  
\end{proof}


 \subsection*{Growth of exponential maps}

  It is well-known (compare \cite[Lemma~2.4]{schleicher-zimmer03}) that all exponential maps (and indeed all
    entire functions of finite order and positive lower order) share the same maximal order of growth of
    their iterates, namely iterated exponential growth. Hence we can use a single model function from this class 
    to describe maximal growth rates of exponential maps. For this purpose, it has become customary to use
    \begin{equation}
        F\colon [0,\infty) \to [0,\infty); \quad t \mapsto e^t-1.
    \end{equation}

   We shall make use of the following elementary fact; compare~e.g.\ Inequalities~(10.1) and~(10.2) in \cite{Rem06}. 

  \begin{lemma}[Iterated exponential growth]\label{lem:exponentialgrowth}
   Fix $K\geq 1$ and $a\in\C$. Then
    \begin{equation}\label{eqn:Finequality}
          F(\re z - 1 ) + K \leq \lvert f_a(z) \rvert \leq F(\re z + 1) - K 
     \end{equation}
    for all $z\in\C$ with $\re z \geq \ln(1 + 2(\lvert a \rvert + K))$.

  Furthermore, for all $R \geq \max(3, \ln(1 + 2(\lvert a \rvert + K)))$ and all $n\geq 1$,  
      \[ R < F^n(R-1) + K \leq M^n(R,f_a) \leq F^n(R+1) - K.\]
  \end{lemma} 
  \begin{proof}
     To prove~\eqref{eqn:Finequality}, set $r \defeq \re z$. Then 
     \begin{align*}   
           F(r-1) + K &< e^{r-1} + K = \frac{1}{e} e^r + K 
                 \leq \frac{1}{2} e^r + K  \\
     & = e^r - \frac{e^r}{2} + K  \leq e^r - \lvert a \rvert 
              = \lvert e^z\rvert - \lvert a \rvert \leq \lvert e^z + a\rvert = 
\lvert f_a(z)\rvert \end{align*}
      by assumption on $r$. 
     Similarly, 
\[ F(r+1) - K = e\cdot e^r - K - 1  \geq 2e^r - K - 1 \geq e^r +  \lvert a \rvert \geq  \lvert f_a(z)\rvert. \]

   Now let $R$ be as in the second claim. Then, by~\eqref{eqn:Finequality}, 
      \begin{equation}\label{eqn:Mestimate}
           F(R+ 1 ) - K \geq M(R, f_a) \geq \lvert f_a(R) \rvert \geq F(R-1)+K \geq F(R-1) > R,
      \end{equation}
    where we use the fact that $e^{R-1} > R+1$ for $R\geq 3$.    Applying~\eqref{eqn:Mestimate} inductively, we see that
      \[ M^{n+1}(R,f_a) + K \leq F(M^n(R,f_a) + 1 ) \leq F(M^n(R,f_a) + K)\leq \dots  \leq 
                           F^{n+1}(R+1), \]
     and analogously for the lower bound. 
  \end{proof}
\begin{rmk}
   It follows, in particular, that $R(f_a)\leq R$ for $R$ as in the statement of the lemma. 
    (Recall that $R(f_a)$ was defined in~\eqref{eqn:R0}.) However, in fact 
     \begin{equation}\label{eqn:Rf_a}
       R(f_a) = 0 
     \end{equation}
     for all $a\in\C$, which means that $A_R(f_a)$ is defined for all $R>0$, and simplifies our statements in the 
     following.  

    Indeed, if $R + e^{-R}<-\re a$, then 
      \[ M(R,f_a) \geq -\re f_a(-R) = -\re a - e^{-R} > R. \]
     On the other hand, if $R + e^{-R}\geq -\re a $, then
      \[ M(R,f_a) \geq \re f_a(R) = e^R + \re a \geq e^R - R - e^{-R} = 2\sinh(R) - R. \]
      Since $\sinh(R)>R$ for $R>0$, we have $M(R,f_a) > R$ in this case also, as claimed. 
\end{rmk}

 A simple consequence of Lemma~\ref{lem:exponentialgrowth}, which is crucial to
    our proof of Theorem~\ref{thm:main}, is the following. Suppose that
   a starting point $z$ has large real part; then we know by the above that $\lvert f_a(z)\rvert$ is very large. If we know that 
    $\lvert \im f_a(z) \rvert$ and $-\re f_a(z)$ are comparatively small compared to this value, then clearly $\re f_a(z)$ is again
    very large and positive. Under suitable hypotheses, we can continue inductively and conclude that $z$ escapes (quickly) to infinity.
    Again, this is a well-known argument in the study of exponential maps; compare e.g.\ 
     \cite[Proof of Theorem~4.4]{schleicher-zimmer03}.
    We shall use it in the following form. 

 \begin{cor}[Continued growth] \label{cor:growth}
   Let $a\in\C$ and $\mu \geq 0$. Then there is $K\geq \mu+2$ such that the following holds for all $z\in\C$ with 
      $\re z \geq K$. If $n\geq 0$ is such that 
    \begin{equation}\label{eqn:boundsforgrowth}
      \max\bigl ( -\re f_a^k(z) , \lvert \im f_a^k(z) \rvert \bigr) \leq F^k(\mu) \end{equation}
     for $0\leq k < n$, then 
       \[ \lvert f_a^n(z) \rvert \geq F^n(\re z -2). \]
 \end{cor}
 \begin{proof}
  Set $K\defeq \max\{2 + \ln(5+\lvert a \rvert), \mu +2\}$.
    Suppose that $r\defeq \re z \geq K$, and that $n$ is as in the statement of the corollary. 
    Observe that the claim is trivial for $n=0$. 
     We shall prove, by induction on $n\geq 1$, the stronger claim 
       \begin{equation}\label{eqn:induction} \lvert f_a^n(z) \rvert \geq  F^n(r-2) + F^n(\mu)+2. \end{equation}
     
     Let $n\geq 1$, and suppose that the claim holds for smaller values of $n$. 
      Then $$\re f_a^{n-1}(z)\geq F^{n-1}(r-2)+2\geq r \geq K.$$ This is true trivially for $n=1$, and
      by the inductive hypothesis~\eqref{eqn:induction} and assumption~\eqref{eqn:boundsforgrowth} for $n>1$. 
      By Lemma~\ref{lem:exponentialgrowth}, 
    \[  \lvert f_a^n(z) \rvert \geq F(\re f_a^{n-1}(z)-1) + 2  \geq
                         2 F^n(r-2) + 2 \geq F^n(r-2) + F^n(\mu) + 2,    \]
     as claimed.
 \end{proof}

\subsection*{Separation}

  Recall that two points $a,b\in X$ are \emph{separated} in a metric space $X$
      if there is a closed and open (``clopen'') set  $U\subset X$ that contains $a$ but not $b$. We shall use
      the following simple lemma only in the case where $X \subset \Ch$ and $x= \infty$.

  \begin{lemma}\label{lem:totseparatedCh}
    Let $X$ be a metric space and $x \in X$. Suppose that $A\defeq X \setminus \{x\}$ is totally separated. Assume furthermore that every point of $A$ is separated from $x$ in
    $X$. Then $X$ is totally separated. 
 \end{lemma}
 \begin{proof}
   Let $a,b\in X$ with $a\neq b$. If one of $a$ and $b$ is equal to $x$, then by assumption $a$ and $b$ are separated in $X$. 
    Otherwise, let $U$ be a clopen set in $A$ containing $a$ but not $b$, and let $V$ be a clopen set in $X$ containing $a$ but not $x$. 
   Then $V$ is open, but not necessarily closed, in $A$.

   Set $W\defeq U\cap V$; then $W$ is open in $X$ and closed in $A$. Furthermore, 
     $x\notin \overline{V} \supset \overline{W}$, and hence $W$ is also closed in $X$. So $W$ is 
     a clopen set of $X$ containing $a$ but not $b$, and the proof is complete.     
 \end{proof}

\subsection*{Separation and spiders' webs} 
 
 We shall use the following notion.

 \begin{dfn}[Separation in $\Ch$]\label{defn:separationCh}
  If $x,y\in\Ch$, we say that $E\subset \Ch$ \emph{separates} $x$ from $y$ 
      if $x$ and $y$ are separated in
      $(\C\setminus E)\cup \{x,y\}$. 

    Analogously, $E$ separates a set $X\subset \Ch$ from a point $y\in\Ch$ if
      there is a clopen set $U\subset (\C\setminus E)\cup X \cup \{y\}$ containing $X$ but not $y$. 
\end{dfn}

      We now prove Theorem~\ref{thm:spiderswebs}, in the following slightly more precise version. 

 \begin{theorem}[Characterisation of spiders' webs]\label{thm:spiderswebsprecise}
   Let $E\subset\C$ (connected or otherwise). The following are equivalent.
     \begin{enumerate}[(a)]
        \item There is a sequence of domains $G_n$ as in the definition of a spider's web.\label{item:spidersweb}
        \item $E$ separates every compact set in $\C$ from $\infty$.\label{item:compactseparation}
        \item $E$ separates every finite point $z$ from $\infty$.  \label{item:pointseparation}
     \end{enumerate}

   Suppose now that one of these equivalent conditions holds. Then the following are equivalent:
    \begin{enumerate}[(1)]
        \item $E$ is connected (that is, $E$ is a spider's web);\label{item:fullspidersweb}
         \item there is a dense collection of unbounded connected subsets of $E$;\label{item:unboundedcomponents}
         \item $E\cup\{\infty\}$ is connected. \label{item:connectedwithinfinity}
    \end{enumerate}
 \end{theorem}
 \begin{proof}
    Clearly,~\ref{item:compactseparation} implies
       \ref{item:pointseparation}. Furthermore, 
     \ref{item:spidersweb} implies~\ref{item:compactseparation}. Indeed, let $G_n$ be the domains
      from the definition of a spider's web. If $X\subset \C$ is compact and $n$ is sufficiently large that
       $X\subset G_n$, then $(G_n\setminus E)\cup X$ is a clopen subset of $(\C\setminus E)\cup X\cup\{\infty\}$, as required. 

    Now suppose that~\ref{item:compactseparation} holds. We claim  that for every nonempty, compact and connected
     $K\subset\C$, there is a 
      bounded simply-connected domain $G=G(K)$ with $K\subset G$ and $\partial G\subset E$.

Indeed,
     as $E$ separates  $K$ and $\infty$, by definition there is a relatively closed and  open subset
         \[U'\subset A \defeq (\C\setminus E)\cup  K \cup \{\infty\} \] 
      such  that $K\subset U'\subset\C$.  
      Let  $U\subset\Ch$ be open such that  $U'=U\cap  A$. Since $U'$ is relatively closed in $A$, 
      we see that  $U$ is bounded and $\partial U\subset \Ch \setminus A \subset E$. 

      Now let $V$ be the connected component
      of $U$ containing $K$, and let $G=G(K)$ be the fill of $V$. (That is, $G$ consists of $V$ 
     together with all bounded complementary components). Clearly $\partial G\subset \partial U\subset E$. 

    So we can define a sequence of simply-connected domains by letting
      $K_0$ be the disc $\overline{D(0,1)}$, and defining inductively $G_j\defeq G(K_j)$ and
      $K_{j+1}\defeq \overline{D(0,j)} \cup \overline{G_j}$. The domains $G_j$ satisfy the  requirements in  the definition of
       a spider's web, so~\ref{item:spidersweb} holds. 
    
    Finally, suppose~\ref{item:pointseparation} holds. Let $K\subset\C$ be a compact set. Then for every
     $x\in K$ there is a bounded open set $U\subset\C$ such that $x\in U$ and $\partial U \subset E$.
   
    Since $K$ is compact, there are $k\in\N$ and $U_1,\dots,U_k$ as above such that $K\subset U\defeq \bigcup U_j$. 
       Clearly 
      \[ \partial  U \subset \bigcup \partial U_j  \subset E, \]
      and $U$ is bounded. So $\partial U$  separates  $K$ from $\infty$. 

  This completes the proof of the equivalence of the three 
     conditions~\ref{item:spidersweb} to~\ref{item:pointseparation}. 
   For the final statement, first observe that~\ref{item:fullspidersweb} $\Rightarrow$~\ref{item:unboundedcomponents}
     $\Rightarrow$~\ref{item:connectedwithinfinity} for all unbounded sets $E$. 

    Clearly the opposite implications do not hold in general, so suppose now 
     that~\ref{item:spidersweb} holds, and that
      $E$ is disconnected. We must show that $E\cup\{\infty\}$ is also disconnected. 
 
   Let  $(G_n)$ be the sequence of domains from~\ref{item:spidersweb}. 
     If  $U$ and $V$
      are disjoint nonempty clopen subsets of $E$ with  $U\cup V = E$, then at least one of these sets, say $U$,
      must contain $\partial G_n$ for infinitely many $n$. Choose such $n$ sufficiently large that 
      also  $V\cap G_n\neq \emptyset$; then it follows that $V\cap G_n$ is a bounded clopen subset of $E$, 
      and hence is also a nonempty and nontrivial proper clopen subset of $E\cup\{\infty\}$.
  \end{proof}

\begin{rmk}
  In~\ref{item:pointseparation}, it is crucial to require separation for all  finite points,
   not just those in  $\Ch\setminus E$. That is, $E$ being a spider's web is a stronger condition  than
   requiring that the quasicomponent of $\infty$ in $\Ch\setminus E$ is a singleton. 
   (Recall that the quasicomponent of a point  $x$ in  a metric space $X$ consists of all points of $X$ not
    separated from  $x$.) 

   This is true even in the case where  $\Ch\setminus E$ is totally separated. Indeed, let  $A\subset\Ch$ be a connected
    set containing $0$ and $\infty$, and having an explosion  point at $0$.
    Then $E\defeq \{0\}\cup \Ch\setminus A$ is not a spider's web, as $0$ is not separated from  infinity in 
     $A =  (\Ch\setminus E)\cup  \{0\}$,  but $\Ch\setminus E =  A\setminus\{0\}$ is  totally separated.  
\end{rmk}

\section{The exponential family}
\label{section_exp}
  
 Theorem~\ref{thm:main} will follow easily from the following result. Recall the definition of
     separation in $\Ch$ (Definition~\ref{defn:separationCh}).  

 \begin{theorem}[Separation using fast escaping points]\label{thm:fastseparation}
   Let $f_a(z) = e^z+a$, and assume that $a\in F(f_a)$. Then, for all $R > 0$ 
    and all $z_0\in \C$, $z_0$  is separated 
    from infinity by $A_R(f_a)\cup F(f_a)$. 
 \end{theorem}
 \begin{proof}
%
    Let $U_1$ be the component of $F(f_a)$ containing $a$, and let $\eps = e^{-c}$ be small enough such that
      $D\defeq \overline{D(a,\eps)}\subset U_1$. 
     By  Proposition~\ref{prop:attractingparabolic}, there is an arc $\sigma\subset U_1$ connecting $D$ to $\infty$, 
     intersecting $D$ only in the finite endpoint, and 
     along which real parts tend to infinity. 

   Consider the closed set 
     \[ \mathcal{M} \defeq f_a^{-1}(D\cup \sigma). \] 
    (See Figure~\ref{subfig:M}.) 
    Then $\mathcal{M}$ consists of the closed left half-plane $\mathcal{L}\defeq\{z\colon \text{Re}(z)\leq -c\}$, together with countably many arcs connecting this half-plane to infinity. Each of these
    arcs is a component of $f_a^{-1}(\sigma)$, and hence they are all $2\pi i \Z$-translates of each other. Furthermore, 
     each of these preimage components is bounded in the imaginary direction 
    (since the argument is bounded along $\sigma$). 

   Let $(S_j)_{j=-\infty}^{\infty}$ denote the complementary components of $\mathcal{M}$, 
     and set 
     \[ \delta \defeq \sup_{z,w\in S_0}{\lvert \im z - \im w\rvert} = \sup_{z,w\in S_j}{\lvert \im z - \im w\rvert}
             \text{ for all  $j\in\Z$}. \]

  Now let $R>0.$ Note that it follows from (\ref{eqn:Rf_a}) that $M(r,f_a) >r$, for $r >0$. Moreover, for $0<R'<R$ we have that $A_R(f_a) \cup F(f_a) \subset A_{R'}(f_a)\cup F(f_a)$. Hence we can increase $R$, if necessary, to ensure that 
            \[ R>\max\bigl(\lvert z_0\rvert ,c, 3, \ln(1 + 2(\lvert a \rvert + \delta))\bigr). \] 
   Set $D_k \defeq D(0,M^k(R,f_a))$ for $k\geq 0$, and consider the set 
  \[ X \defeq \bigcap_{k\geq 0} \left( f_a^{-k}\left(\C\setminus D_k\right) \cup \bigcup_{j=0}^{k-1} f_a^{-j}(\mathcal{M})\right).\]
   That is, if  $x\in X$ and $k\geq 0$, then either
     $\lvert f_a^k(z)\rvert \geq M^k(R,f_a)$, or the orbit of $z$ has entered $\mathcal{M}$ before time $k$. 
Note that since $|z_0|<R$ and $X \subset \C \setminus D(0,R)$ we have $z_0\notin X$. Also 
    $X\subset A_R(f_a)\cup F(f_a)$ by definition. It is thus enough to show that 
     $z_0$ is separated from $\infty$ by $X$. Since $X$ is closed as an intersection of closed sets,
     this is equivalent to showing that the connected component $V$ of $\C\setminus X$ containing 
     $z_0$ is bounded. 

   By the definition of $X$, the modulus of the forward images of any point in $V$ 
     must fall behind
     the growth given by $M^k(R,f_a)$ in order to be able to enter $\M$; i.e. 
    \[ V \cap f_a^{-n}(\M) \subset (\C\setminus X) \cap f_a^{-n}(\M) \subset 
       \bigcup_{k=0}^n f_a^{-k}(D_k) \]
   for all $n\geq 0$.
   Since $f_a(D_k) \subset D_{k+1}$ by definition, we have 
     \begin{equation}\label{eqn:imagesofV}
        f_a^n(V) \cap \M \subset D_n 
      \end{equation}
 for all $n\geq 0$. (See Figure~\ref{fig:proof}.) 

 \begin{figure}
    \def\svgwidth{\textwidth}
     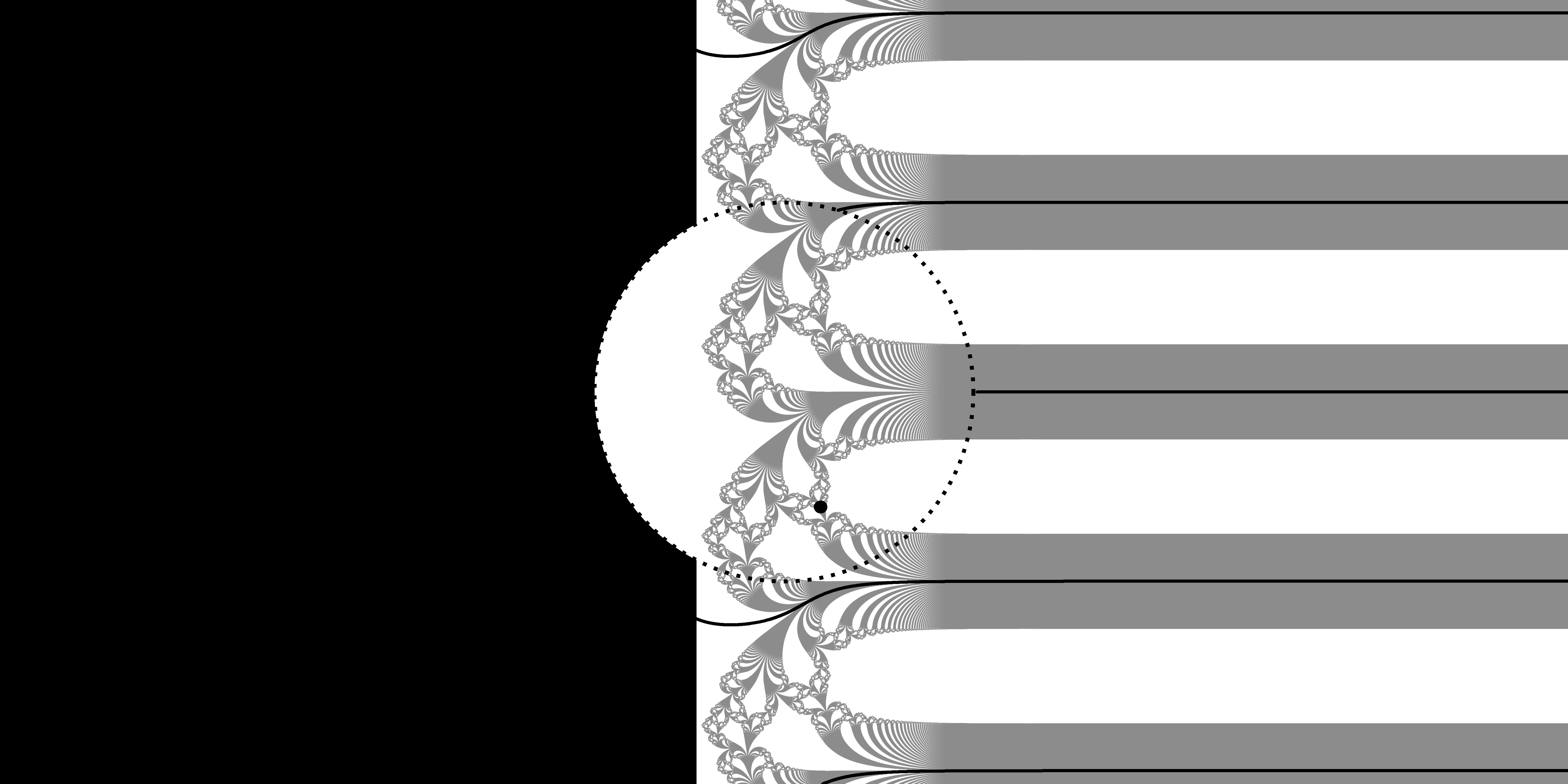
    \caption{\label{fig:proof}Illustration of the proof of Theorem~\ref{thm:fastseparation}. The domain $f_a^n(V)$ cannot intersect the set shown in black,
       which is the part of  $\M$ from Figure~\ref{subfig:M} that does not lie in the disc $D_n = D(0,M^n(f_a,R))$. (The boundary of this disc is shown as a dotted line.) 
       Since $f_a^n(z_0)\in D_n\cap f_a^n(V)$,
       any strip $S_j$ which does not intersect $D_n$ also cannot intersect $f_a^n(V)$.} 
 \end{figure}
 
  Let $n\geq 0$. Since $\lvert z_0\rvert < R$, we have $f_a^n(z_0) \in f_a^n(V)\cap D_n$. 
    If $S_j$ is a complementary
     component of $\M$ which does not intersect $D_n$, then $f^n(z_0)\in f_a^n(V)\setminus S_j$, and
    furthermore $f_a^n(V)\cap \partial S_j = f_a^n(V)\cap \M\cap \overline{S_j} =  \emptyset$ by~\eqref{eqn:imagesofV}. Hence 
    $S_j\cap f_a^n(V)=\emptyset$ (see Figure~\ref{fig:proof}). Thus~\eqref{eqn:imagesofV} can be
     reformulated as
\begin{equation}\label{eqn:imagesofV2}
   f_a^n(V) \subset D_n \cup \bigcup\bigl\{ S_j\colon S_j\cap D_n \neq\emptyset\bigr\}. \end{equation}

  Now let $z\in V$. Then there is a minimal $N\geq 0$ 
     such that $f_a^{N}(z) \in D_{N}$. 
     By~\eqref{eqn:imagesofV} we have 
      $f_a^n(z)\notin \M$ for $n < N$, and hence 
      \begin{equation}\label{eqn:negativerealpartsinV} \re f_a^n(z) > -c > -R > - F^n(R+1). \end{equation}
     Moreover, using Lemma~\ref{lem:exponentialgrowth} and choice of $R$, we see 
       from~\eqref{eqn:imagesofV2} that      
   \begin{equation}\label{eqn:imbound}
      \left\lvert \im f_a^n(z) \right\rvert \leq M^n(R,f_a) + \delta \leq F^n(R+1)
     \end{equation}
   for all $n$. Hence
     \begin{equation} \label{eqn:upperrealbound}
       \re z \leq \max(R+3,K),
     \end{equation}  
   where $K$ is as in Corollary~\ref{cor:growth}, applied to  $\mu = R+1$. 
    Indeed, otherwise we could conclude from Corollary~\ref{cor:growth} that 
     \[ \lvert f_a^{N+1}(z)\rvert \geq F^{N+1}(\re z -2 ) > F^{N+1}(R+1), \]
    which contradicts the fact that 
    $\lvert f_a^{N+1}(z)\rvert < M^{N+1}(R,f_a) < F^{n+1}(R+1)$ by choice of $N$ 
     and Lemma~\ref{lem:exponentialgrowth}.

   In conclusion, $\im z$ is bounded in $V$ by~\eqref{eqn:imbound}, while $\re z$ is bounded from below
      by~\eqref{eqn:negativerealpartsinV} and from above by~\eqref{eqn:upperrealbound}. So $V$ is bounded, as required and so for any $R>0$ we can separate $z_0$ from infinity by $A_R(f_a) \cup F(f_a)$.
 \end{proof} 
Theorem \ref{thm:fastseparation} implies that $A_R(f_a)\cup F(f_a)$ has the structure of a spider's web.
\begin{cor}[Spiders' webs]\label{cor:fastsw}
  Let $f_a(z) = e^z+a$, where $a\in F(f_a)$. Then, for all $R>0$,  $A_R(f_a)\cup~F(f_a)$ is a spider's web.
\end{cor}
\begin{proof}
Let $E\defeq A_R(f_a)\cup F(f_a)$. By 
   Theorem~\ref{thm:fastseparation}, property~\ref{item:pointseparation} in
   Theorem~\ref{thm:spiderswebsprecise} is satisfied for $E$.
   Since 
   $J(f_a)$ is nowhere dense, $F(f_a)$ is dense in $E$. Moreover, it follows from 
  Proposition~\ref{prop:attractingparabolic} that every connected component of $F(f_a)$ contains an arc along which the real part tends to infinity, and so is unbounded. Hence it follows from the second part of 
   Theorem~\ref{thm:spiderswebsprecise} that $E$ is a spider's web. 
\end{proof}
We can also deduce Theorem \ref{thm:main}, in the following more precise version. 
\begin{cor}[Total separation]\label{cor:totsep}
  Let $f_a(z) = e^z+a$, where $a \in F(f_a)$. Then the set 
    $(E(f_a)\setminus A_R(f_a))\cup\{\infty\}$ is totally separated for all $R\geq 0$.

    In particular, $\meandering(f_a)\cup \{\infty\}$ is totally separated. 
\end{cor}
\begin{proof}
  The set $E(f_a)\subset J(f_a)$ 
     is totally separated by \cite[Theorem 1.7]{alhabib-rempe15}. Hence the first claim follows from
 Theorem~\ref{thm:fastseparation} and Lemma~\ref{lem:totseparatedCh}. 

   By Proposition~\ref{prop:hairs}, the set $A(f_a)$ contains all points on hairs
     (as well as some endpoints). Hence
    \[ \meandering(f_a) = J(f_a) \setminus A(f_a)=E(f_a)\setminus A(f_a) \subset E(f_a) \setminus A_R(f_a), \]
     and thus $\meandering(f_a)\cup\{\infty\}$ is totally
     separated by the first claim. 
\end{proof}
\begin{rmk}
  On the other hand, it follows from
    the construction in \cite[Remark 4.6]{alhabib-rempe15} that 
    the connected component of $\infty$ in $(E(f_a)\cap A_R(f_a))\cup\{\infty\}$ is nontrivial for all $R$. 
\end{rmk}

\section{Exponential maps whose singular value lies in the Julia set} \label{sec:julia}
In this section we remark upon the case where $a \in J(f_a)$. In this case, $F(f_a)$ is either empty or
  consists of a cycle of Siegel discs, together with their preimages.
  We can still apply our method of proof from the previous section to obtain the following result. 

  \begin{theorem}[Points staying away from the singular value]\label{thm:avoidsingularvalue}
   Let $a\in\C$ with  $a\in J(f_a)$, let $\eps>0$ and $R>0$. Let $S$ denote the set of points $z\in \C$ with 
    \begin{equation}\label{eqn:avoidsingularvalue} \inf_{n\geq 0} |f_a^n(z) - a| > \eps.\end{equation}
    Then every point $z_0\in S$ is separated from infinity by $A_R(f_a)\cup (\C\setminus S)$. 
  \end{theorem}
 \begin{remark}
   If $a\in F(f_a)$, then there exists $\eps >0$ such that every point in $J(f_a)$ trivially satisfies~\ref{eqn:avoidsingularvalue}. Hence, by
     Theorem~\ref{thm:fastseparation}, 
      we can remove the hypothesis ``$a\in J(f_a)$'' in Theorem~\ref{thm:avoidsingularvalue}
     if we replace ``$z\in\C$'' by ``$z\in J(f_a)$'' (or, even, ``$z$ not belonging to an attracting or parabolic basin'').
 \end{remark}
  \begin{proof}
    Write $\eps = e^{-c}$ and $D \defeq \overline{D(a,\eps)}$. Since $D$ intersects the Julia set, it follows by the blowing-up property of the Julia set (see e.g.\ \cite[Section~2]{Berg}, \cite[Lemma~2.1]{rippon-stallard09}) that there is 
       a point $\zeta \in  D$ and  $n\geq 1$ such that $f_a^n(\zeta)\in \mathcal{L}=\{z\colon \text{Re}(z)\leq -c\}$.
       Now connect $f_a^n(\zeta)$ to infinity by an arc $\sigma_0$ in $\mathcal{L}$, chosen such that
       $\sigma_0$ avoids $f_a^k(a)$ for $k=0,\dots,n-1$. As in Proposition~\ref{prop:attractingparabolic}, 
         the connected component $\sigma_1$ of
       $f_a^{-n}(\sigma_0)$ containing $\zeta$ is an arc in $\C\setminus S$ whose real parts tend to $+\infty$.
       By deleting the maximal piece of $\sigma_1$ connecting $\zeta$ to $\partial D$, we obtain an
     arc $\sigma$ connecting $D$ to infinity. 
       Set $\mathcal{M}\defeq f_a^{-1}(D\cup \sigma)$, and continue as in the proof of Theorem~\ref{thm:fastseparation}. 

     So we again obtain a closed set $X$, and see that the connected component $V$ of $\C\setminus X$ containing
      $z_0$ is bounded. All points in the set $X$ either belong to  $A_R(f_a)$, or otherwise
      map into the left half-plane $\mathcal{L}=\{z: \text{Re}(z)\leq -c\}$, and thus into $D$. So
      $X\subset A(f_a) \cup (\C\setminus S)$, and we obtain the desired conclusion. 
  \end{proof}

 In the case where the Fatou set is non-empty,
   Theorem~\ref{thm:avoidsingularvalue} immediately implies the following result,
   proved in \cite{rempe-04}.

   \begin{theorem}\label{thm:siegel}
     Let $f_a(z)= e^z+a$, and assume that $f_a$ has a cycle  $U_1 \mapsto \dots \mapsto U_n \mapsto U_1$ ($n\geq 1$) of Siegel discs such that no
       $\partial U_j$ contains the singular value $a$.  

      Then all $U_j$ are bounded. 
   \end{theorem}

  (The proof of Theorem~\ref{thm:siegel} in \cite{rempe-04} also relies on using a similar set as $\mathcal{M}$ 
    above; compare \cite[Figure~2]{rempe-04}.)

 To conclude the section, note that
   it follows immediately from Proposition~\ref{prop:hairs}~\ref{item:nonlanding2} 
  that, when $a\in J(f_a)$,  the structure of the full set of non-escaping points in $J(f_a)$ can look
  very different from the case where $a\in  F(f_a)$. 

  \begin{cor}[Large meandering sets]\label{cor:largemeandering}
    Let $a\in\C$ be a parameter such that, for the exponential map $f_a$, the 
     omitted value $a$ is either an endpoint or on a hair. Then $ (J(f_a)\setminus I(f_a))\cup \{\infty\}$ is connected.  
     In particular, neither $I(f_a)$ nor $A(f_a)$ is a spider's web. 
  \end{cor}

  Recall that, when $f_a$ is postsingularly finite, the radial Julia set coincides with the set of non-escaping points (Corollary \ref{cor:radialJ}(c)).
   Thus we obtain, in particular, the statement concerning $\radial(f_a)$ made in  the introduction. 
 
 \section{Fatou's function}
 \label{section_fatou}

  We now turn to studying Fatou's function 
    \begin{equation}\label{eqn:fatousfunction}
      f\colon \C\to \C; \qquad z\mapsto z+1+e^{-z}. \end{equation}
  Observe that  $f$ is semiconjugate to 
      \begin{equation}\label{eqn:fatousemiconjugate}
        h\colon \C\to \C; \qquad \zeta \mapsto e^{-1}\zeta e^{-\zeta} \end{equation}
      via the correspondence $\zeta = g(z) \defeq \exp (-z)$. 
  
   It is well-known that $J(f)$ is a Cantor bouquet while $F(f)$ consists of a single domain in which the
    iterates tend to infinity; see e.g.\ 
     \cite[Theorem 1.3]{alhabib-rempe15}. Hence it again makes sense to speak about \emph{hairs} and
     \emph{endpoints} of $J(f)$. Moreover, also by \cite[Theorem 1.3]{alhabib-rempe15}, all
     points on hairs belong to  $A(f)$; in  other words, all points in $\meandering(f)$ are endpoints. 
     We refer to~\cite{Evdoridou16} for further background. 
    
    As noted in the introduction, the following implies
    \cite[Theorems 1.1 and 5.2]{Evdoridou16}. 

 \begin{theorem}[Fatou's web revisited]\label{thm:fatousweb}
    Let $f$ be Fatou's function~\eqref{eqn:fatousfunction}. Then the set $F(f)\cup A(f)$ is a spider's web, and 
      its complement $\meandering(f)$~-- i.e., the set of meandering endpoints of $f$~-- together with infinity forms a totally separated set. 
  \end{theorem}  

  We could prove this theorem by mimicking the proof of Theorem~\ref{thm:maindisjointtype}. Instead,
    let us show that the latter in fact implies Theorem~\ref{thm:fatousweb}, using the semiconjugacy between
    $f$ and $h$, together with known
    (albeit non-elementary) results. 

  \begin{proposition}[Structure of $h$]\label{prop:h}
    Let $h$ be as in~\eqref{eqn:fatousemiconjugate}, 
     and let $f_{-2}$ be the exponential map $z\mapsto e^z-2$. Then there is 
     a homeomorphism $\phi\colon \C\to\C$ such that $\phi(J(f_{-2})) = J(h)$,
     $\phi(I(f_{-2})) = I(h)$ and  $\phi(A(f_{-2})) = A(h)$. 

     In particular, $\meandering(h)\cup \{\infty\}$ is totally separated, and the complement of $\meandering(h)$ is a spider's web. 
  \end{proposition} 
   \begin{proof}
     The function $h$ is conjugate to $\tilde{h}\colon w\mapsto (w+1) e^w -1$, via 
        $w = -\zeta - 1$, so it is enough to prove the claim  for $\tilde{h}$.
         It is shown in \cite{rempe-09} that there is a quasiconformal homeomorphism  $\tilde{\phi}$ that   
         conjugates $f_{-2}$ and $\tilde{h}$ on their Julia sets (see \cite[Figure~1]{rempe-09}); in particular,
      $\tilde{\phi}(J(f_{-2})) = J(\tilde{h})$ and $\tilde{\phi}(I(f_{-2})) = I(\tilde{h})$. Since quasiconformal
       maps are H\"older, it also follows that $\tilde{\phi}(A(f_{-2})) = A(\tilde{h})$. 

   Indeed, recall from Lemma~\ref{lem:exponentialgrowth} that $z\in A(f_{-2})$ if and only if 
     for some, and hence all, $T>0$ 
     there is $n_0\geq 0$ such that $\lvert f_{-2}^{n_0+k}(z) \rvert \geq F^k(T)$ for all $k\geq 0$.  
     By a similar calculation, the same is true for $h$ and, in fact, any entire function of positive lower order and
     finite upper order. (Compare e.g.\ \cite[Lemma~3.4]{alhabib-rempe15}.) 

   Now let $z\in I(f_{-2})$ and  $w \defeq \tilde{\phi}(z)\in I(\tilde{h})$, and denote the orbits of these
     points under the correspoding maps by $(z_n)_{n\geq 0}$ and  $(w_n)_{n\geq 0}$, respectively. 
      Since $\tilde{\phi}$ is H\"older, there is $\alpha >1$
        such that
      \[ \lvert w_n \rvert^{1/\alpha} \leq \lvert z_n \rvert \leq  \lvert w_n \rvert^{\alpha}. \] 
   
    By an elementary calculation, 
         $F(T^{\alpha}) > F(T)^{\alpha}$ for all sufficiently large $T$. It follows that
        $z\in A(f_{-2})$ if and only if $w\in A(\tilde{h})$, as claimed. 
        
Recall that $\meandering(E)= J(E)\setminus A(E)$ by definition, for any entire function $E$; so also 
   $\tilde{\phi}(\meandering(f_{-2}))= \meandering(\tilde{h})$. The final claim now follows from
    Corollaries \ref{cor:fastsw} and \ref{cor:totsep}.
   \end{proof} 

  \begin{proof}[Proof of Theorem~\ref{thm:fatousweb}]
    Let $g\colon z\mapsto \exp(-z)$ be the semiconjugacy between $f$ and $h$, so
       $g\circ f = h\circ g$. 
     By~\cite[Theorems~1 and~5]{bergweiler-hinkkanen99}, we have
       $g(J(f)) \subset J(h)$ and 
      $g^{-1}(A(h))~\subset~A(f)$. 
     (It is well-known that actually
      $g(J(f)) = J(h)$, which also follows from \cite{bergweiler-hinkkanen99} and 
      the fact that $A(f)\subset J(f)$, but we do not require this here.) 

    Hence 
       \[ \meandering(f) = J(f) \setminus A(f) \subset g^{-1}( J(h) \setminus A(h)) = g^{-1}(\meandering(h)).  \] 
     Since $h(x) = x\cdot e^{-(x+1)} < x$ for $x>0$, we see that 
         $[0,\infty)\in F(h)$. Taking inverse branches of $g$ on the slit plane $\C\setminus [0,\infty)$, 
        we therefore see that $\meandering(f)$ is contained in a countable collection of homeomorphic
        copies of  $\meandering(h)$, which are mutually separated from each other by the horizontal lines
        whose imaginary parts are even multiplies of $\pi$. 

Since $A(h)\cup F(h)$ is a spider's web (Proposition~\ref{prop:h}) it follows from Theorem~\ref{thm:spiderswebsprecise} that it separates every finite point from $\infty$. Hence the above argument implies that $A(f) \cup F(f)$ separates every finite point from $\infty$ and since it is connected it is also a spider's web.
  \end{proof}

\newcommand{\etalchar}[1]{$^{#1}$}
\providecommand{\bysame}{\leavevmode\hbox to3em{\hrulefill}\thinspace}

\providecommand{\href}[2]{#2}


\begin{thebibliography}{BDH{\etalchar{+}}00}

\bibitem[AO93]{aartsoversteegen}
J.M. Aarts and L.G.. Oversteegen, \emph{The geometry of {J}ulia sets}, Trans.
  Amer. Math. Soc. \textbf{338} (1993), no.~2, 897--918.

\bibitem[AR16]{alhabib-rempe15}
N.~{Alhabib} and L.~{Rempe-Gillen}, \emph{Escaping endpoints explode}, Comput.
  Methods Funct. Theory, doi:10.1007/s40315-016-0169-8 (2016).

\bibitem[BD00]{bhattacharjee-devaney-2000}
R.~Bhattacharjee and R.L. Devaney,
  \emph{\href{http://math.bu.edu/people/bob/papers/ranjit.ps}{Tying hairs for
  structurally stable exponentials}}, Ergodic Theory Dynam. Systems \textbf{20}
  (2000), no.~6, 1603--1617.

\bibitem[BDH{\etalchar{+}}00]{BDHRGH}
C.~Bodel\'on, R.L. Devaney, M.~Hayes, G.~Roberts, L.R. Goldberg, and J.H.
  Hubbard, \emph{Dynamical convergence of polynomials to the exponential}, J.
  Differ. Equations Appl. \textbf{6} (2000), no.~3, 275--307.

\bibitem[Ber93]{Berg}
W.~Bergweiler, \emph{Iteration of meromorphic functions}, Bull. Amer. Math.
  Soc. (N.S.) \textbf{29} (1993), no.~2, 151--188.

\bibitem[BH99]{bergweiler-hinkkanen99}
W.~Bergweiler and A.~Hinkkanen, \emph{On semiconjugation of entire functions},
  Math. Proc. Cambridge Philos. Soc. \textbf{126} (1999), no.~3, 565--574.

\bibitem[BJR12]{baranski-jarque-rempe12}
K.~Bara\'nski, X.~Jarque, and L.~Rempe, \emph{Brushing the hairs of
  transcendental entire functions}, Topology Appl. \textbf{159} (2012), no.~8,
  2102--2114.

\bibitem[BR84]{baker-rippon}
I.N. Baker and P.J. Rippon, \emph{Iteration of exponential functions}, Ann.
  Acad. Sci. Fenn. Ser. A I Math. \textbf{9} (1984), 49--77.

\bibitem[Dev84]{devaney-84}
R.L. Devaney, \emph{Julia sets and bifurcation diagrams for exponential maps},
  Bull. Amer. Math. Soc. (N.S.) \textbf{11} (1984), no.~1, 167--171.

\bibitem[DG87]{DevaneyGoldberg1987}
Robert~L. Devaney and Lisa~R. Goldberg, \emph{Uniformization of attracting
  basins for exponential maps}, Duke Math. J. \textbf{55} (1987), no.~2,
  253--266. 

\bibitem[DK84]{devaney-krych84}
R.L. Devaney and M.~Krych, \emph{Dynamics of {${\rm exp}(z)$}}, Ergodic Theory
  Dynam. Systems \textbf{4} (1984), no.~1, 35--52.

\bibitem[Evd16]{Evdoridou16}
V.~Evdoridou, \emph{Fatou's web}, Proc. Amer. Math. Soc. \textbf{144} (2016),
  no.~12, 5227--5240.

\bibitem[Fat26]{fatou}
P.~Fatou, \emph{Sur l'it\'eration des fonctions transcendantes enti\`eres},
  Acta Math. \textbf{47} (1926), 337--370.

\bibitem[FRS08]{frs}
M.~F{\"o}rster, L.~Rempe, and D.~Schleicher, \emph{Classification of escaping
  exponential maps}, Proc. Amer. Math. Soc. \textbf{136} (2008), no.~2,
  651--663.

\bibitem[Kar99]{karpinskaparadox}
Bogus{\l}awa Karpi{\'n}ska, \emph{Hausdorff dimension of the hairs without
  endpoints for {$\lambda\exp z$}}, C. R. Acad. Sci. Paris S\'er. I Math.
  \textbf{328} (1999), no.~11, 1039--1044.

\bibitem[May90]{mayer90}
J.C. Mayer, \emph{An explosion point for the set of endpoints of the {J}ulia
  set of {$\lambda\exp(z)$}}, Ergodic Theory Dynam. Systems \textbf{10} (1990),
  no.~1, 177--183.

\bibitem[Maz29]{Mazurkiewicz1929}
S.~Mazurkiewicz, \emph{Sur les points accessibles des continus
  ind\'{e}composables}, Fundamenta Mathematicae \textbf{14} (1929), no.~1,
  107--115.

\bibitem[McM00]{mcmullenradial}
Curtis~T. McMullen, \emph{Hausdorff dimension and conformal dynamics {II}.
  {G}eometrically finite rational maps}, Comment. Math. Helv. \textbf{75}
  (2000), no.~4, 535--593.

\bibitem[Nad92]{nadler}
S.B. Nadler, Jr., \emph{Continuum theory. {A}n introduction}, Monographs and
  Textbooks in Pure and Applied Mathematics, vol. 158, Marcel Dekker Inc., New
  York, 1992.

\bibitem[Rem04]{rempe-04}
L.~Rempe, \emph{On a question of {H}erman, {B}aker and {R}ippon concerning
  {S}iegel disks}, Bull. London Math. Soc. \textbf{36} (2004), no.~4, 516--518.

\bibitem[Rem06]{Rem06}
\bysame, \emph{Topological dynamics of exponential maps on their escaping
  sets}, Ergodic Theory Dynam. Systems \textbf{26} (2006), no.~6, 1939--1975.

\bibitem[Rem07]{Rem07}
\bysame, \emph{On nonlanding dynamic rays of exponential maps}, Ann. Acad. Sci.
  Fenn. Math. \textbf{32} (2007), no.~2, 353--369.

\bibitem[Rem09a]{Rem09}
\bysame, \emph{Hyperbolic dimension and radial {J}ulia sets of transcendental
  functions}, Proc. Amer. Math. Soc. \textbf{137} (2009), no.~4, 1411--1420.

\bibitem[Rem09b]{rempe-09}
\bysame, \emph{Rigidity of escaping dynamics for transcendental entire
  functions}, Acta Math. \textbf{203} (2009), no.~2, 235--267.

\bibitem[RRS10]{Devaney-Hairs}
L.~Rempe, P.J. Rippon, and G.M. Stallard, \emph{Are {D}evaney hairs fast
  escaping?}, J. Difference Equ. Appl. \textbf{16} (2010), no.~5-6, 739--762.

\bibitem[RS05]{rippon-stallard-05}
P.J. Rippon and G.M. Stallard, \emph{On questions of {F}atou and {E}remenko},
  Proc. Amer. Math. Soc. \textbf{133} (2005), no.~4, 1119--1126.

\bibitem[RS08]{rempe-schleicher08}
L.~Rempe and D.~Schleicher, \emph{Bifurcation loci of exponential maps and
  quadratic polynomials: local connectivity, triviality of fibers, and density
  of hyperbolicity}, Holomorphic dynamics and renormalization, Fields Inst.
  Commun., vol.~53, Amer. Math. Soc., Providence, RI, 2008, pp.~177--196.

\bibitem[RS09a]{rempe-schleicher09}
\bysame, \emph{Bifurcations in the space of exponential maps}, Invent. Math.
  \textbf{175} (2009), no.~1, 103--135.

\bibitem[RS09b]{rippon-stallard09}
P.J. Rippon and G.M. Stallard, \emph{Escaping points of entire functions of
  small growth}, Math. Z. \textbf{261} (2009), no.~3, 557--570.

\bibitem[RS12]{rippon-stallard12}
\bysame, \emph{Fast escaping points of entire functions}, Proc. London Math.
  Soc. (3) \textbf{105} (2012), no.~4, 787--820.

\bibitem[SZ03a]{schleicher-zimmer03}
D.~Schleicher and J.~Zimmer, \emph{Escaping points of exponential maps}, J.
  London Math. Soc. (2) \textbf{67} (2003), no.~2, 380--400.

\bibitem[SZ03b]{schleicher-zimmer03-periodic}
\bysame, \emph{Periodic points and dynamic rays of exponential maps}, Ann.
  Acad. Sci. Fenn. Math. \textbf{28} (2003), 327--354.

\bibitem[Urb95]{urbanskiconical}
Mariusz Urba\'nski, \emph{On some aspects of fractal dimensions in higher
  dimensional dynamics}, Proceedings of the workshop ``{P}roblems in higher
  dimensional dynamics'', Preprint SFB 170 G\"ottigen, vol.~3, 1995,
  pp.~18--25.

\bibitem[UZ03]{urbanski-zdunik03}
M.~Urba\'nski and A.~Zdunik, \emph{The finer geometry and dynamics of the
  hyperbolic exponential family}, Michigan Math. J. \textbf{51} (2003), no.~2,
  227--250.

\end{thebibliography}
 \end{document}